\documentclass[american, 11pt]{article}
\usepackage[utf8]{inputenc}
\usepackage{amsmath, amssymb, amsthm, faktor}
\usepackage{enumerate}
\usepackage{tikz-cd}
\usepackage[curve]{xypic}
\usepackage{graphicx}
\usepackage{faktor}
\usepackage{enumitem}
\usepackage{authblk}

\usepackage{xcolor}

\newcounter{warn}[page]
\newcommand{\danger}{${\color{red}\triangle}\llap{\raisebox{.3ex}%
{\tiny!\hspace{1.45ex}}}$}
\newcommand{\warning}[1]{%
\raisebox{.01em}[0em]{\danger\ifnum\value{warn} > 1%
\tiny\bf\arabic{warn}\fi}%
\marginpar{\color{red}\tiny{\ifnum\value{warn} > 1\tiny\bf\arabic{warn}:\fi}\tiny #1}%
\stepcounter{warn}}

\newtheorem{theorem}{Theorem}[section]
\newtheorem{definition}[theorem]{Definition}

\newtheorem{proposition}[theorem]{Proposition}
\newtheorem{lemma}[theorem]{Lemma}

\theoremstyle{remark}
\newtheorem{remark}{Remark}

\renewcommand{\emptyset}{\varnothing}

\newcommand{\R}{\mathbb{R}}

\renewcommand{\S}{\mathbb{S}}

\makeatletter%
\newcommand{\tto@args}[2] {{\displaystyle\mathop{\rightrightarrows}^{#2}_{#1}}}%
\newcommand{\tto}[1][{}{}]{\tto@args #1}%
\makeatother%

\makeatletter
\DeclareRobustCommand*{\rightfaktor}[3][]
{
   { \mathpalette{\rightfaktor@impl@}{{#1}{#2}{#3}} }
}
\newcommand*{\rightfaktor@impl@}[2]{\rightfaktor@impl#1#2}
\newcommand*{\rightfaktor@impl}[4]{
   \settoheight{\faktor@zaehlerhoehe}{\ensuremath{#1#2{#3}}}%
   \settoheight{\faktor@nennerhoehe}{\ensuremath{#1#2{#4}}}%
      \raisebox{-0.5\faktor@zaehlerhoehe}{\ensuremath{#1#2{#3}}}%
      \mkern-4mu\diagdown\mkern-5mu%
      \raisebox{0.5\faktor@nennerhoehe}{\ensuremath{#1#2{#4}}}%
}
\makeatother

\newcommand{\xxto}[1]{\xrightarrow{#1}}
\newcommand{\glueup}{\sharp^{\bullet}}
\newcommand{\gluelo}{\sharp_{\bullet}}
\newcommand{\eval}{\mathrm{ev}}
\newcommand{\Loops}{\mathop{\Omega}\nolimits}
\newcommand{\FullLoops}{\Loops'}
\newcommand{\Orbits}{\mathcal{P}}
\newcommand{\CovOrbits}{\tilde{\Orbits}}
\newcommand{\Crit}{\mathrm{Crit}}
\newcommand{\M}{\mathcal{M}}

\newcommand{\sublevel}[3][]{{#2}^{#1\leq#3}}       
\newcommand{\zipped}[2]{\lfloor#1\rfloor_{#2}}
\newcommand{\zip}[2]{\zeta_{#1}^{#2}}

\newcommand{\Action}{\mathcal{A}}

\newcommand{\Cov}[1]{\tilde{#1}}
\newcommand{\J}{\mathcal{J}}  

\begin{document}

\title{Floer-Novikov fundamental group and small flux symplectic isotopies}%

\author[1]{Jean-François Barraud}%
\author[2]{Agnès Gadbled}%

\affil[1]{Institut de Mathématiques de Toulouse ; UMR5219 Université de
  Toulouse ; CNRS.}

\affil[2]{Mathématiques Orsay; Université Paris Saclay ; CNRS.}

\maketitle

\begin{abstract}
  Floer theory relates the dynamics
  of Hamiltonian isotopies and the homology of the ambient manifold.
  It was extended to similarly relate the dynamics of symplectic isotopies
  and the Novikov homology associated to their flux. We discuss this picture
  regarding the fundamental group, and prove that when the flux is not too big,
  the associated Novikov fundamental group is generated by Floer
  moduli spaces associated to closed orbits of the symplectic isotopy.
\end{abstract}

\section{Introduction and main statement}

The celebrated Floer theory, introduced by Floer in \cite{Floer3,Floer2} as a
tool to prove the Arnold conjecture, is designed to study fixed points of
Hamiltonian isotopies or intersections of Lagrangian submanifolds under
deformation by such isotopies, from an homological point of view. Among
many other development of this theory, several authors
\cite{Sikorav1986},\cite{LeOno1995},\cite{Mihai2009},\cite{Agnes2009}
extended his ideas to the case of symplectic (non Hamiltonian) isotopies,
showing that the theory still makes sense, if the homology of the ambient
manifold is replaced by the Novikov homology associated to the flux of the isotopy. 

The goal of this paper is to study the same question from the fundamental
group point of view. In the Hamiltonian setting, the Floer theory is rich
enough to recover generators of the fundamental group of the ambient
(closed and monotone) manifold as explained in \cite{FloerPi1}. On the
other hand, to a degree $1$ cohomology class $[\alpha]$ on a closed
manifold $M$ and a choice of integration cover $\Cov{M}$ for $[\alpha]$ is
naturally associated a group $\pi_{1}(\Cov{M},[\alpha])$, that generalizes
the usual fundamental group to the Novikov setting, as explained in
\cite{NovikovPi1}. It is then natural to expect that the Floer
construction adapts from the Hamiltonian to the symplectic case,
replacing the fundamental group by the Novikov fundamental group.

\medskip

The main theorem of this paper is to show that it is indeed the case for
isotopies that have a small enough flux.

Consider a closed monotone symplectic manifold $(M,\omega)$
and a non degenerate symplectic isotopy $(\phi_{t})_{t\in[0,1]}$. Let
$X_{t}=\frac{d\phi_{t}}{dt}$ be  the vector field generating this
isotopy. The $1$-form
$$
\alpha_{\phi} = \int_{0}^{1}\omega(X_{t},\cdot)dt
$$
is then closed, and its cohomology class $[\alpha_{\phi}]$ is called the
flux of the isotopy (or its Calabi invariant \cite{LeOno1995}). This
cohomology class only depends on the homotopy class of the path
$(\phi_{t})$ with fixed ends.

Choose now an integration cover $\Cov{M}$ for $[\alpha_{\phi}]$. There might be
several possible choices, from the minimal one to the universal cover,
and we fix one once for all (the resulting group we are about to define
will depend on this choice, and each choice defines a different version
of the invariant, just like in the case of Novikov homology).

Pick an $\omega$ compatible almost complex structure $J$, which we allow
to depend on two parameters $(s,t)\in[0,1]\times\S^{1}$, and suppose it
is chosen generic, meaning that all the relevant Floer theoretic moduli
spaces are cutout transversely. Then a group $\Loops(\phi,J)$ can be
built out of theses moduli spaces with the following property~:

\begin{theorem}\label{thm:EvalOntoPi1}
  Let $(M,\omega)$ be a closed monotone symplectic manifold and
  $(\phi_{t})$ a (non Hamiltonian) symplectic isotopy as above. If the flux
  $[\alpha_{\phi}]$ is small enough, then there is a surjective map
  $$
  \begin{tikzcd}
  \Loops(\phi,J) \arrow[r,twoheadrightarrow]& \pi_{1}(\Cov{M},[\alpha_{\phi}])
  \end{tikzcd}
  $$
  from $\Loops(\phi,J)$ to the Novikov fundamental group associated to
  the flux of the isotopy.
\end{theorem}

\begin{remark}
  Notice that a small flux does not mean a small isotopy~:
  Hamiltonian isotopies have a vanishing flux, but can still
  be arbitrary large. 
\end{remark}


\begin{remark}
  The construction relies on curves that are typically used to define the
  PSS morphism \cite{PSS} between Floer and Morse homologies, and the key
  point is an energy/depth estimate for such curves to provide a control
  of such curves in the Novikov setting. In particular, the construction
  below could also provide a PSS morphism between the Floer to Morse Novikov
  homologies, as long as the flux of the symplectic isotopy used on the
  Floer side is small enough.
\end{remark}

As an obvious corollary, we obtain a way to detect fixed points of
symplectic isotopies.

\begin{proposition}
  In the situation of theorem  \ref{thm:EvalOntoPi1}, if
  $\pi_{1}(\Cov{M},[\alpha_{\phi}])\neq 1$, then $\phi$ has fixed points.
\end{proposition}

\begin{remark}
  More explicit examples are easier to derive from the Lagrangian
  version...
\end{remark}

\subsection{Moduli spaces}
Let $\J$ be the space of $\omega$-compatible almost complex structures on
$M$ that depend on two parameters $(s,t)\in[0,1]\times\S^{1}$, and such
that $J(0,t)$ is constant.

The main ingredient in the construction of the
group of Floer loops $\Loops(\phi,J)$ are the PSS-like moduli spaces
$$
\M(y,\emptyset)
$$
associated to an almost complex structure $J\in\J$, i.e. moduli spaces of
maps $u:\R\times\S^{1}\to \Cov{M}$, with finite energy, that are solutions of the
``truncated'' Floer equation
\begin{equation}
  \label{eq:augmentation_equation}
\frac{\partial u}{\partial s}+J_{\chi(s),t}(u)
\Big(\frac{\partial u}{\partial t}-\chi(s)X_{t}(u)\Big) = 0.
\end{equation}
Here, the cutoff function $\chi(s)$ is a smooth function such that
$\chi(s)=1$ for $s\leq -1$ and $\chi(s)=0$ for $s\geq 0$, and the
almost complex structure $J$ is in fact the lift of $J$ to $\tilde{M}$.

Solutions of this equation with finite energy do have limits at the ends,
which are
\begin{itemize}
\item
  a $1$-periodic orbit $y$ of $X$ at $-\infty$ (in $\tilde{M}$),
\item
  an point $p\in\Cov{M}$ at $+\infty$.
\end{itemize}

More precisely, we consider two lifts of the periodic orbits~: first we
consider $X$ as a vector field on $\Cov{M}$, and the set $\Orbits$ of its
contractible periodic orbits consists of all the lifts of the
contractible periodic orbits in $M$. Second, we consider the covering
$\CovOrbits$ of $\Orbits$ obtained by considering discs bounded by periodic orbits
under the equivalence relation~:
$$
\gamma\sim\gamma' \iff
\omega(\gamma)=\omega(\gamma') \text{ and }
\mu_{CZ}(\gamma)=\mu_{CZ}(\gamma'),
$$
where $\mu_{CZ}$ denotes the Conley Zehnder index. From now on, we will
avoid stressing the use of these coverings all along the paper, and when
speaking of a ``periodic orbits of $X$'', we will in fact refer to an
element in $\CovOrbits$.

In particular, in the situation above, the curve $u$ defines a disc
bounded by the periodic orbit at $-\infty$, and we will see the limit $y$
as an element of $\CovOrbits$ rather than $\Orbits$.

For convenience, we will use the following shifted index rather than the
Conley Zehnder index on $\CovOrbits$~:
$$
|y| = \mu_{CZ}(y)+n
$$
(where $n=\frac{1}{2}\dim(M)$). Then, for a generic choice of $J$, the
moduli space $\M(y,\emptyset)$ is a smooth manifold and
$$
\dim \M(y,\emptyset)=|y|.
$$

We are interested in the connected components of such $1$ dimensional
moduli spaces.

\paragraph{Compactification of bounded ends}

Consider a connected component of $\M(y,\emptyset)$ for some periodic
orbit $y$ such that $|y|=1$. We are interested in the case when it is not
closed (i.e.compact without boundary). Since it is a one dimensional
manifold, it has two ends, each of them being homeomorphic to a half line
$[0,\infty)$.
\begin{itemize}
\item
  If the energy is bounded on an end~: it is said to be a bounded end, and
  it can be compactified by adding a broken configuration
  $(u,v)\in\M(y,x)\times\M(x,\emptyset)$ through an intermediate orbit
  $x$ with $|x|=0$.
\item
  Otherwise it is said to be an unbounded end.
\end{itemize}

From now on, $\M(y,\emptyset)$ will denote the moduli space with
compactified bounded ends.

The minimal requirement to control the unbounded ends is to show that
some notion of depth is proper on $\M(y,\emptyset)$, which is the object
of the next section.

\begin{figure}
  \centering
  \includegraphics[scale=.5]{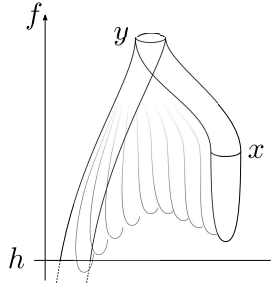}
  \caption[Step]{Bounded and unbounded ends in $\M(y,\emptyset)$.}
  \label{fig:step}
\end{figure}

\section{Energy/depth estimates}\label{sec:EnergyDepthEstimates}

Let $H_{t}:\Cov{M}\to\R$ be a Hamiltonian on $\Cov{M}$ generating
the isotopy i.e. such that
$$
dH _{t}= -\omega(X_{t},\cdot).
$$

We use the deformation lemma from \cite[lemma 2.1, p.157]{LeOno1995}, to
modify the isotopy, keeping its ends fixed, so that the cohomology class
$[-\omega(X_{t},\cdot)]$ is in fact constant equals to the flux 
$[\alpha]$.

We pick a primitive $\Cov{M}\xrightarrow{f}\R$ of $\pi^{*}\alpha$. We
will refer to the values of $f$ at a point $p$ as its height in $\Cov{M}$
with respect to $[\alpha]$.

Notice that $dH_{t}-df$ descends to $M$ as an exact form, so that there
is a constant $K$ such that
\begin{equation}
  \label{eq:distHf}
  \Vert H-f \Vert_{\infty}\leq K.
\end{equation}

\subsection{Average depth estimate}

Let $u\in\M(y,\emptyset)$ be a solution of the truncated Floer equation
\eqref{eq:augmentation_equation} as above. 

Recall that the energy of $u$ is
$$
E(u)=\iint ||\frac{\partial u}{\partial s}||^2dsdt.
$$

The following straightforward computation:
\begin{align*}
  E(u)
  &=\iint \omega(\frac{\partial u}{\partial s},J\frac{\partial u}{\partial s})dsdt\\
  &=\iint \omega(\frac{\partial u}{\partial s},\frac{\partial u}{\partial t})dsdt - 
    \iint \omega(\frac{\partial u}{\partial s},X)\chi(s) dsdt\\
  &=\iint u^*\omega dsdt - 
    \iint dH_{t}(\frac{\partial u}{\partial s})\chi(s) dsdt\\
  &=\iint u^*\omega dsdt - 
    \int [H_{t}(u)\chi(s)]_{s=-\infty}^{s=+\infty} + \iint H_{t}(u)\chi'(s)ds dt\\
  &=\iint u^*\omega dsdt + 
    \int H_{t}(y(t))dt + \iint H_{t}(u)\chi'(s)ds dt\\
  &= \Action(y) - \iint H_{t}(u)|\chi'(s)|ds dt
    \stepcounter{equation}
    \tag{\theequation}\label{eq:Energy_DepthAverage}
\end{align*}
shows that for a fixed action $\Action(y)$, the energy is related to the
average of $H$ in the cutoff region $\{-1\leq s\leq0\}\times\S^{1}$ with respect to
the measure $|\chi'(s)|dsdt$, and hence to the average depth of $u$ in
this region.

However, we want to evaluate one dimensional moduli spaces as paths in
$\Cov{M}$ and keep control of the depth along such paths~: this requires
point-wise estimates that cannot be directly derived from the above
average estimate in general. 

\medskip


The object of the next section is to obtain such a point-wise estimate
when the flux is not too big, based on the monotonicity principle and the
Schwarz lemma  for pseudo-holomorphic curves observed by Gromov in
\cite{Gromov1985}.

\medskip

Before proceeding, we need to upgrade the average depth estimate
\eqref{eq:Energy_DepthAverage} on the region $-1\leq s\leq 0$ into
estimates that are pointwise with respect to $s$.

Namely, let
$$
m(s) = \int_{0}^{1}H_{t}(u(s,t))dt
$$
Then $m'(s) = \int_{0}^{1}\omega(\frac{\partial u}{\partial s}, X)dt$,
and for $s_{0}\leq s_{1}$ we have~:
\begin{align*}
  m(s_{1})-m(s_{0})
  &= \int_{s=s_{0}}^{s=s_{1}}\int_{t=0}^{t=1}
    \omega(\frac{\partial u}{\partial s}, X)dtds
  \\
  &= \iint_{s_{0}\leq s\leq s_{1}}
    <\frac{\partial u}{\partial t}, X> - \chi \Vert X\Vert^{2}dtds
  \\
  &\leq
    \left(\iint_{s_{0}\leq s\leq s_{1}} \Vert\frac{\partial u}{\partial t}\Vert^{2}dsdt\right)^{1/2}
    \left(\iint_{s_{0}\leq s\leq s_{1}} \Vert X \Vert^{2} ds dt \right)^{1/2}
  \\
  &\leq \sqrt{E(u)}\  \sqrt{s_{1}-s_{0}} \ \Vert X \Vert_{\infty}
\end{align*}
where $\Vert X \Vert_{\infty} = \sup\{\omega(X_{t}(p),J_{s,t}X_{t}(p)),
(s,t)\in[0,1]\times\S^{1}, p\in M\}$. On the other hand, from
\eqref{eq:Energy_DepthAverage}, there is at least one $s_{0}\in[-1,0]$
such that $m(s_{0}) = \Action(y)-E(u)$. This implies that for $s\geq 0$
we have~:
\begin{equation}
  \label{eq:EnergyEstimateAts=1}
  \int_{0}^{1} H_{t}(u(s,t))dt\leq -E(u)+
  \Vert X \Vert_{\infty}\sqrt{s+1}\sqrt{E(u)} + \Action(y)
\end{equation}
and hence, letting $\sigma_{u}=\frac{E(u)}{4\Vert X \Vert^{2}}-1$, we have
for all $s\in[0,\sigma_{u}]$~:
\begin{equation}\label{eq:EnergyEstimateOnTransitionAnnulus}
  \forall s\in[0,\sigma_{u}]:\
  \int_{0}^{1} H_{t}(u(s,t))dt\leq -\frac{E(u)}{2}+\Action(y).
\end{equation}


\subsection{Depth estimate at $+\infty$.}

%
%

On the line $\R[\alpha]$, we pick a generator $[\alpha_{0}]$, and a
primitive $f_{0}$ of a form $\alpha_{0}$ in this class. We define
$\lambda$ as
\begin{equation}
  \label{eq:lambda}
  [\alpha] = \lambda [\alpha_{0}]
\end{equation}

Let $\delta_{0}$ be a non trivial period of $[\alpha_{0}]$ (i.e. a positive
real number such that there is some $g\in\pi_{1}(M)$ with
$\alpha_{0}(g)=\delta_{0}$), and consider the slice
$$
S = f_{0}^{-1}([0,\delta_{0}])\subset\Cov{M}
$$
Notice that $\Cov{M}$ is a union of copies of this slice under deck
transformations.

\medskip

The main remark we'll make use of to relate the depth at $u(+\infty)$ and
the energy of our curves is that holomorphic curves need a minimal energy
to go through a slice. This is a direct consequence of the monotonicity
principle for $J$-holomorphic curves described by Gromov \cite{Gromov1985}.
However, since our curves are not everywhere holomorphic, we will also
need to study the behavior of curves for which transition region between
the Floer and the purely holomorphic areas is stretched across a large
height.

\subsubsection{Energy of holomorphic discs with low boundary and  high center}

First recall the crucial monotonicity
principle for $J$ holomorphic curves, due to M. Gromov.
\begin{lemma}\label{lem:Monotonicity}
  Given an $\omega$ compatible almost complex structure $J$, there are
  constants $r_{0}$ and $C$ such that, for every $r\in(0,r_{0}]$, for
  every point $x\in\Cov{M}$ and every $J$-holomorphic map $u: \Sigma\to
  \Cov{M}$ defined on a Riemann surface $\Sigma$ such that
  \begin{itemize}
  \item
    $u(\Sigma)\subset B(x,r)$,
  \item
    $u(\partial\Sigma)\subset\partial B(x,r)$
  \item
    $u$ goes through the center of the ball,
  \end{itemize}
  then $\iint_{\Sigma} u^{*}\omega\geq C r^{2}$
\end{lemma} 

A direct consequence of the monotonicity principle is that holomorphic
curves that go across a slice $S$ of $\Cov{M}$ have a symplectic area
bounded from below. 

\begin{lemma}\label{lem:EnergyDepthEstimateForDiscs}
  There are constants $K_{1}$ and $K_{2}$ (depending only on $M$, 
  $\omega$, $J$, $\Xi_{0}$ but not on $\lambda$, $X$ or $H$) such that,
  for every $u\in\M(y,\emptyset)$ and every $s_{0}\geq 0$~:
  \begin{equation}
    \label{eq:EnergyEstimateForHighCenter}
    E(u)\geq \frac{K_{1}}{\lambda} \Big(f(u(+\infty))-\max_{\{s=s_{0}\}}(f(u))\Big) - K_{2}
  \end{equation}
\end{lemma}


\begin{proof}
Recall we let $S=f_{0}^{-1}([0,\delta_{0}])$, and consider
$S'=f_{0}^{-1}([\frac{\delta_{0}}{3},\frac{2\delta_{0}}{3}])$. Recall
that the almost complex structure $J(s,t)$ is in fact constant for $s\geq
0$, and consider a radius $r_{0}$ and a constant $C$ associated to this
almost complex structure by lemma\ref{lem:Monotonicity}. Pick a radius
$R\leq r_{0}$ such that at every point $p\in S'$, the $R$-ball at $p$ is
contained in $S$.

The holomorphic disc $u_{|_{\{s\geq s_{0}\} } }$ has to cross at least $N$
copies of $S$ where
$$
N = \left\lfloor
\frac{f_{0}(u(0))-\max_{\{s=s_{0}\}}(f_{0}(u))}{\delta_{0}}
\right\rfloor-2.
$$

For each copy $S_{i}$ of $S$, pick an $R$-ball centered at a point
$p_{i}=u(z_{i})\in S'_{i}$ for $z_{i}\in D$~: from lemma
\ref{lem:Monotonicity}, we obtain
$$
E(u)\geq N CR^{2}
$$
which leads to the desired estimate. 
\end{proof}

\subsubsection{Upper bounds for the height on the transition annulus.}

For a curve $u\in\M(y,\emptyset)$, recall the notation $
\sigma_{u}=\frac{E(u)}{4\Vert X \Vert^{2}}-1 $ from
\eqref{eq:EnergyEstimateOnTransitionAnnulus}, and consider the annulus
$$
A_{u}=[0,\sigma_{u}]\times\S^{1}.
$$

We now want to prove that the annulus $A_{u}$ does indeed contain a loop
$\{s=cst\}$ that does not go above a deep level~:
\begin{lemma}\label{lem:TransitionUpperBound}
  There is a constant $K_{3}$ (depending on $M$, $\omega$, $J$,
  $\phi$) such that, for every $u\in\M(y,\emptyset)$, there is some
  $s_{0}\geq 0$ such that
  $$
  \max_{t\in\S^{1}} \{f(u(s_{0},t))\}\leq \Action(y)-\frac{E(u)}{2} + K_{3}
  $$
\end{lemma}

Let
\begin{equation}
  \label{eq:magnitude}
  \Delta_{f}(s)=
  \max_{t\in\S^{1}}f(u(s,t))-\min_{t\in\S^{1}}f(u(s,t))
\end{equation}
and consider a point $s_{0}\in[0,\sigma_{u}]$ where this magnitude is
minimal~:
$$
\Delta_{f}(s_{0}) = \min_{s\in[0,\sigma_{u}]}\Delta_{f}(s).
$$
Recall that $|f-H|$ is uniformly bounded on $\Cov{M}$. Using
\eqref{eq:EnergyEstimateOnTransitionAnnulus} in the last line of the
following estimates
\begin{align*}
  \max_{t\in\S^{1}}f(u(s_{0},t))
  &\leq
    \int_{0}^{1} f(u(s,t))dt+\Delta_{f}(s_{0})\\
  &\leq
    \int_{0}^{1} H_{t}(u(s,t))dt+\Vert f-H \Vert_{\infty}+\Delta_{f}(s_{0})\\
  &\leq
    \Action(y)-\frac{E(u)}{2}+\Vert f-H \Vert_{\infty}+\Delta_{f}(s_{0}),
\end{align*}
we obtain that the proof of lemma \ref{lem:TransitionUpperBound} reduces
 to the following lemma~:
\begin{lemma}    \label{lem:TransitionHeightBound}
  There is a uniform constant $K_{3}$ (depending on $M$, $\omega$,
  $J$,$[\alpha]$ and $\phi$), such that for all $u\in\M(y,\emptyset)$~:
  \begin{equation}
    \label{eq:TransitionHeightBound}
    \min_{s\in[0,\sigma_{u}]}\Delta_{s}f(u)\leq  K_{3}
  \end{equation}
\end{lemma}

We will need the classical Gromov-Schwarz lemma, which is a consequence
of the monotonicity principle. The following form is again picked from
\cite[p.181]{AL1994}~:
\begin{lemma}\label{lem:GromovSchwarzLemma}
  Given an $\omega$-compatible almost complex structure $J$ that may
  depend on a parameter in the unit disc, there is a constant $C$ such
  that, for every map $u:D\to \Cov{M}$ defined on the unit disc such that
  \begin{itemize}
  \item
    $u$ is $J$-holomorphic in $D$,
  \item
    $\iint_{D}u^{*}\omega\leq a_{0}$,
  \end{itemize}
  then $\Vert d_{0}u  \Vert \leq C$.
\end{lemma}

\begin{proof}[Proof of lemma \ref{lem:TransitionHeightBound}]
  For all
  $s\in[0,\sigma_{u}]$, we have
  $$
  \max_{t\in\S^{1}}(f(u(s,t)))
  -\min_{t\in\S^{1}}(f(u(s,t)))
  \geq \Delta_{f}(s_{0}),
  $$
  so that above each $s\in[0,\sigma_{u}]$ there is a point
  $z_{s}=s+it_{s}$ such that
  $$
  \Vert du(z_{s}) \Vert
  \geq
  \frac{\Delta_{f}(s_{0})}{\Vert df \Vert_{\infty}}
  $$

  Let $r=\frac{\Vert df \Vert_{\infty}}{\Delta_{f}(s_{0})}\, C$ (where $C$ is
  the constant appearing in the Gromov-Schwarz lemma
  \ref{lem:GromovSchwarzLemma}), and consider the $r$-subdivision
  $(s_{1},\dots,s_{N})$ given by
  $$
  s_{k}=k\, r,\quad 1\leq k\leq N
  \ \text{ with }\
  N=\left\lfloor \frac{\sigma_{u}}{r}\right\rfloor
  $$
  The associated points $z_k=s_{k}+it_{s_{k}}$ are such that~:
  $$
  r\Vert du(z_{k}) \Vert\geq C
  $$
  and hence, from lemma \ref{lem:GromovSchwarzLemma}:
  $$
  \iint_{D(z_{k},r)}u^{*}\omega\geq a_{0}.
  $$

  Because the discs $D(z_{k},r)$ are all disjoint, we obtain
  \begin{align*}
    E(u)
    &\geq Na_{0}\geq (\frac{\sigma_{u}}{r}-1)a_{0}
      =
      \left(
      \frac{\sigma_{u}\, \Delta_{s_{0}}f(u)}{C\Vert df \Vert_{\infty}}
      -1
      \right)a_{0}
  \end{align*}
  Recalling that $\sigma_{u}=\frac{E(u)}{4\Vert X \Vert^{2}}$, this means
  that
  $$
  \Delta_{s_{0}}f(u)\leq \frac{4C\Vert X \Vert^{2}\Vert df \Vert_{\infty}}{a_{0}}
  \left(\frac{E(u) +a_{0}}{E(u)}\right),
  $$
  which leads to the desired result when $E(u)\geq1$. On the other hand,
  if $E(u)\leq 1$, $u$ belongs to a compact subset of the moduli space,
  on which the maximal and minimal height, and hence a fortiori
  $\Delta_{s_{0}}f(u)$, have to be bounded.  
\end{proof}

\subsubsection{Depth estimate at $+\infty$.}

Recall from \eqref{eq:lambda} the definition of $\lambda$ by the relation
$$
[\alpha]= \lambda[\alpha_{0}].
$$

\begin{lemma}\label{lem:EnergyDepthEstimate}
  There is a positive constant $A$ depending only on $M$, $\omega$, $J$
  and $[\alpha_{0}]$ (but not on $\lambda$), and a constant $B$ that may
  depend also on $\lambda$ and $\phi$, such that for all maps
  $u\in\M(y,\emptyset)$~:
  $$
  f(u(+\infty)) \leq \Action(y) - (\frac{1}{2}-A\lambda) E(u)
  + B
  $$
  In particular, for $\lambda<\frac{1}{2A}$, we have
  $$
  \lim_{E(u)\to+\infty}f(u(+\infty))=-\infty.
  $$
\end{lemma}

\begin{proof}
  From lemma \ref{lem:TransitionUpperBound}, there is some $s_{0}$ such
  that
  $$
  \max_{t\in\S^{1}} \{f(u(s_{0},t))\}\leq \Action(y)-\frac{E(u)}{2} + K_{3}
  $$
  On the other hand, from lemma \ref{lem:EnergyDepthEstimateForDiscs}, we
  have 
  \begin{align*}
  f(u(+\infty))
  &\leq
  \frac{\lambda (E(u)+K_{2})}{K_{1}}+\max_{t\in\S^{1}}\{f(u(s_{0},t))\}\\
  &\leq
    \Action(y)-\Big(\frac{1}{2}- \frac{\lambda}{K_{1}}\Big) E(u)+ 
    \lambda \frac{K_{2}}{K_{1}} + K_{3},
  \end{align*}
  which ends the proof since $K_{1}$ is independent of $\lambda$.
  \begin{figure}
    \centering
    \includegraphics[scale=.6]{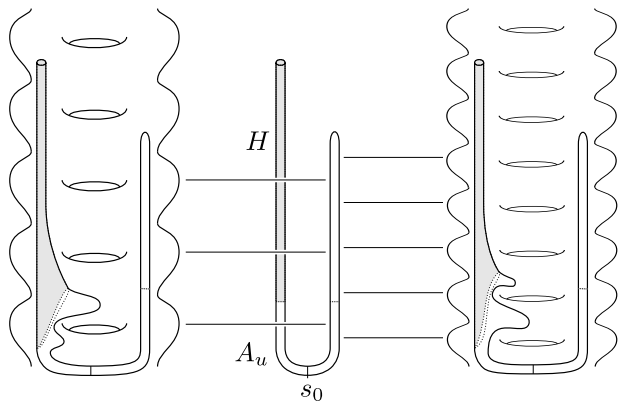}
    \caption{Holomorphic compensation of average Hamiltonian descent costs
      more energy for smaller flux.} 
    \label{fig:energydepth}
  \end{figure}
  
\end{proof}

Notice for future use that the same argument applies when the roles of the
ends in the Floer equation are flipped. We denote by $\M(\emptyset,y)$
the corresponding space. Then we have the following lemma~:

\begin{lemma}\label{lem:CoEnergyDepthEstimate}
  There is positive constant $A$ depending only on $M$, $\omega$, $J$ and
  $[\alpha_{0}]$ (but not on $\lambda$) and a constant $B$ that may depend also
  on $\lambda$ and $\phi$, such that, for all maps $u\in\M(\emptyset,y)$~:
  $$
   f(u(+\infty))\geq \Action(y) + (\frac{1}{2}-\lambda A) E(u)- B.
   $$
   In particular, for $\lambda<\frac{1}{2A}$, we have
   $$
   \lim_{E(u)\to+\infty}f(u(-\infty))=+\infty.
   $$
\end{lemma}

\section{Floer Novikov loops for small flux}

We now restrict to symplectic isotopies that have a small enough flux,
namely such that $\lambda < \frac{1}{2A}$ with the notation of lemma
\ref{lem:EnergyDepthEstimate}, so that the map
$$
\begin{array}{ccc}
\M(y,\emptyset)&\xrightarrow{f\circ\eval}& \R\\
u & \mapsto & f(u(+\infty))  
\end{array}
$$
is proper on all the moduli spaces $\M(y,\emptyset)$. From now on, the
\emph{depth} of a curve $u$ will refer to the depth $f(u(+\infty))$ of the point
$u(\infty)$.

\subsection{Definition of Floer Novikov loops}

The definition of Floer-Novikov loops mimics the definition of
Morse-Novikov loops given in \cite{NovikovPi1}. We let
$$
\M(y,\{ f\geq h\})=\{u\in\M(y,\emptyset), f(u(+\infty))\geq h\}
$$
For a generic choice of the level $h$, this is a one dimensional manifold
with boundary, and its boundary is given either by Floer breaks or by the
condition $f(u(+\infty))=h$.

\begin{definition}
  A Floer-Novikov step relative to $h$ is a connected component of a $1$
  dimensional moduli space $\M(y,\{f\geq h \})$ with non empty boundary,
  endowed with an orientation.
\end{definition}

\begin{remark}
  In this definition all the components of $\M(y,\{f\geq h\})$ are
  considered separately. Another choice would be to concatenate all such
  components that belong to the same component of $\M(y,\emptyset)$. We
  will see below that the two choices eventually lead to the same group.  
\end{remark}
\begin{remark}
  This definition obviously depends on the choice of the function $f$
  used to measure the ``depth'', but it will be rather obvious that the
  resulting defintion of the Floer Novikov fundamental group will not.
\end{remark}

According to its orientation, a step $\sigma$ has a starting and an
ending level, which is either $h$ or $f(v(+\infty))$ if the corresponding
end is a broken configuration $(u,v)\in\M(y,x)\times\M(x,\emptyset)$.

It also has a highest level which is the highest depth $f(u(+\infty))$ over all
curves $u$ in the step.

\begin{definition}
  Fix a level $h\in\R$. Two Floer Novikov steps $\sigma_{1}$ and
  $\sigma_{2}$ are said to be consecutive if either~:
  \begin{itemize}
  \item
    $\sigma_{1}$ ends and $\sigma_{2}$ starts on the level $h$,
  \item
    or $\sigma_{1}$ ends and $\sigma_{2}$ starts with broken
    configurations that involve the same orbit $x$ and the same curve
    $v\in\M(x,\emptyset)$.
  \end{itemize}
  
  A Floer-Novikov loop relative to $h$ is then a sequence of consecutive
  steps, the first starting and the last ending on the level $h$.
\end{definition}

\begin{figure}
  \centering
  \includegraphics[scale=.5]{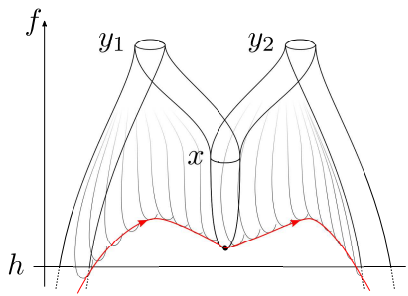}
  \caption{A Floer-Novikov loop relative to level $h$.}
  \label{fig:loop}
\end{figure}

The obvious concatenation rule and the equivalence relation $\sim$
induced by cancellation of the occurrence of two consecutive copies of the
same step with opposite orientations turns the collection of all loops
relative to $h$ into a group, that will be denoted by
$$
\zipped{\Loops(\Cov{M},J,\phi)}{h}.
$$

Moreover, given two levels $h'<h$, loops relative to $h'$ are a fortiori
loops relative to $h$, and there is a natural restriction map
$$
\zipped{\Loops(\Cov{M},J,\phi)}{h'} \xrightarrow{\zip{h'}{h}}
\zipped{\Loops(\Cov{M},J,\phi)}{h}.
$$
Finally, given three levels $h''<h'<h$, we have $
\zip{h'}{h}\circ\zip{h''}{h'}=\zip{h''}{h}$.

\begin{definition}
  Define the group of Floer Novikov loops as
  $$
  \Loops(\Cov{M},J,\phi) = \varprojlim 
  \zipped{\Loops(\Cov{M},J,\phi)}{h}.
  $$
\end{definition}

For convenience, we may omit the dependency on $\Cov{M}$ and $J$ in the
notation.

\begin{remark}
  Since the difference between two choices of height functions is always
  bounded, it is not hard to see that the direct limit process discards
  all dependency on this choice. 
\end{remark}

\subsubsection{Full Floer Novikov steps}

A component $\sigma$ of $\M(y,\emptyset)$, when restricted above a level
$h$, defines a sequence of components in $\M(y,\{f\geq h\})$, i.e. a
sequence of steps, that are obviously consecutive. We call this
concatenation of all the Floer steps that come from the same component of
 $\M(y,\emptyset)$ a \emph{full} Floer Novikov step.

\begin{definition}
  A full Floer step above a given level $h$ is the concatenation of all
  the Floer steps relative to $h$ that belong to the same component
  $\sigma\subset\M(y,\emptyset)$, in the order given by this component.

  We denote by $\zipped{\FullLoops(\phi)}{h}$ the associated space of loops
  (i.e. sequences of consecutive full steps, the first starting and the
  last ending on level $h$) and let
  $$
  \FullLoops(\phi) =
  \varprojlim_{h} \zipped{\FullLoops(\phi)}{h}.
  $$
\end{definition}

Notice that $\zipped{\FullLoops{\phi}}{h}$ is a subgroup of
$\zipped{\Loops{\phi}}{h}$, and the restriction maps induce an inclusion
in the limit.
$$
\FullLoops{\phi}\hookrightarrow\Loops{\phi}.
$$

\begin{proposition}\label{prop:RegularVersusFullLoops}
  The loops groups generated by full or regular Floer steps are the same~:
  $$
  \FullLoops{\phi} = \Loops{\phi}.
  $$
  In other words, using the terminology of \cite{NovikovPi1}, the collection of
  all the components of all the moduli spaces $\M(y,\emptyset)$ generate
  $\Loops(\phi)$ up to deck transformations and completion. 
\end{proposition}

This is a consequence of lemma \ref{lem:Icebergs} below, which itself is
a direct consequence of the properness of the map $u\mapsto
f(u(+\infty))$.

\begin{lemma}\label{lem:Icebergs}
  For every $\Delta^{+}>0$, there is a constant $\Delta^{-}>0$ such that
  for every levels $h$ and $h'$  with $h'\leq h-\Delta^{-}$, and every
  index $1$ orbit $y$ with $\Action(y)\leq h+\Delta^{+}$, two components
  of $\M(y,\{f\geq h\})$ that belong to the same component of
  $\M(y,\emptyset)$ belong to the same component of $\M(y,\{f\geq h'\})$. 
\end{lemma}

\begin{figure}
  \centering
  \includegraphics[scale=.8]{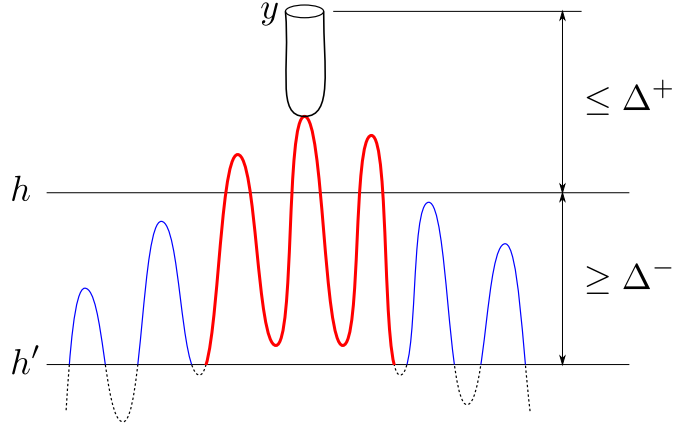}
  \caption{The traces of a component of $\M(y,\emptyset)$ above a level $h$
    eventually join as a single component above sufficiently deep levels $h'$.} 
  \label{fig:Icebergs}
\end{figure}

\begin{proof}
  If this is not the case, we find a constant $\Delta^{+}$ and a sequence
  $h_{n}$, $h'_{n}$, $y_{n}$ such that
  \begin{enumerate}
  \item
    $\Action(y_{n})\leq h_{n}+\Delta^{+}$,
  \item
    $h'_{n}<h_{n}-n$,
  \end{enumerate}
  and two disjoint components $\sigma'_{n,1}$ and $\sigma'_{n,2}$ of
  $\M(y,\{f\geq h'_{n}\})$ that belong to the same component of
  $\M(y,\emptyset)$ such that $\sigma'_{n,i}\cap\M(y,\{f\geq h_{n}\})\neq
  \emptyset$ for $i=1,2$.
  
  Up to a sub-sequence and shifts in $\Cov{M}$, the orbit $y_{n}$ can be
  supposed to be in fact a constant orbit $y$. The sequence $h_{n}$ is
  then bounded from below, and can also be supposed to be constant
  without loss of generality.

  For $i=1,2$, pick some $u_{n,i}\in\sigma'_{n,i}\cap\M(y,\{f\geq
  h_{n}\})$.   Since $\M^{\geq h}(y,\emptyset)$ is compact, both sequences
  $(u_{n,1})$ and $(u_{n,2})$ and can be supposed to converge, and hence
  to be constant.
  
  Since $u_{1}$ and $u_{2}$ belong to the same component of
  $\M(y,\emptyset)$ which is $1$ dimensional, they bound a compact
  segment $[u_{1},u_{2}]$ in $\M(y,\emptyset)$. On the other hand, since
  this segment cannot is not fully contained in $\M(y,\{f\geq h'_{n}\})$
  by assumption, there is a point $v_{n}$ between $u_{1}$ and $u_{2}$
  such that $f(v_{n}(+\infty))< h'_{n}$. In particular
  $$
  \lim f(v_{n}(+\infty))=-\infty,
  $$
  which contradicts the compactness of $[u_{1},u_{2}]$.
\end{proof}

\begin{proof}[Proof of proposition \ref{prop:RegularVersusFullLoops}]
  Pick an element $\gamma\in\Loops(\phi)$, and a level $h$. Consider the
  reduced word representing $\zip{-\infty}{h}(\gamma)$~: it is a finite
  sequence $(\sigma_{1},\dots,\sigma_{k})$ of components of moduli spaces
  $\M(y_{i},\{f\geq h\})$.

  Let $\Action_{\max}=\max\{\Action(y_{i})\}$, and consider the level
  $h'=h-\Delta^{-}$ where $\Delta^{-}$ is the constant provided by lemma
  \ref{lem:Icebergs} when taking $\Delta^{+}=\Action_{\max}-h$.

  Then $\zip{-\infty}{h}(\gamma)=\zip{h'}{h}(\zip{-\infty}{h'}(\gamma))$.
  This means that the components $\sigma_{i}$ are restriction above level
  $h$ of components of $\M(y,\{f\geq h'\})$~: from lemma
  \ref{lem:Icebergs}, this means they are in fact full Floer steps.  
\end{proof}

\subsection{Evaluation}

Notice that the evaluation at $+\infty$, denoted as
$$
\M(y,\emptyset) \xrightarrow{\eval} \Cov{M},
$$
continuously extends to the broken configurations. It turns each
component of $\M(y,\emptyset)$ into a path in $\Cov{M}$ that is well
defined up to parameterization, and for which unbounded ends go to
$-\infty$ in $\Cov{M}$. 

Focusing on subs-levels and passing to homotopy classes, we get rid of
the parameterization ambiguity and obtain a well defined maps
$$
\zipped{\Loops(\Cov{M},J,\phi)}{h}\xrightarrow{\eval}
\zipped{\pi_{1}(\Cov{M},\alpha)}{h}
$$
for every levels $h\in\R$. For $h'<h$ they make the following
diagram commutative~:
$$
\xymatrix{
  \zipped{\Loops(\Cov{M},J,\phi)}{h}\ar[r]^{\eval}&
  \zipped{\pi_{1}(\Cov{M},\alpha)}{h}\\
  \zipped{\Loops(\Cov{M},J,\phi)}{h'}\ar[r]^{\eval}\ar[u]^{\zip{h'}{h}}&
  \zipped{\pi_{1}(\Cov{M},\alpha)}{h'}\ar[u]^{\zip{h'}{h}}
}
$$
In particular, these evaluation maps induce a map in the limit~:
\begin{equation}
  \label{eq:evaluation}
  \Loops(\Cov{M},J,\phi)
  \xrightarrow{\eval}
  \pi_{1}(\Cov{M},\alpha).
\end{equation}

The main result in this paper is the following:
\begin{theorem}
  Consider a non degenerate symplectic isotopy $\phi$ in $M$ and equip
  $M$ with a generic almost complex structure $J\in\J$. If the flux of $\phi$ is small enough, then the evaluation map
  $$
  \Loops(\Cov{M},J,\phi)
  \xrightarrow{\eval}
  \pi_{1}(\Cov{M},\alpha)
  $$
  is onto.
\end{theorem}

The proof of this theorem reduces to prove that any Morse Novikov loop is
homotopic to Floer Novikov loop, which is the object of the next section.

\section{From Morse Novikov to Floer Novikov loops}

In this section, we suppose that the $1$-form picked in the cohomology
class $[\alpha]$ to define the depth function $f$ is Morse, so that the
function $f$ itself is Morse on $\Cov{M}$. Moreover, as $[\alpha]\neq0$, we
can also suppose for convenience that $f$ has no index $0$ critical point.

We also pick a Riemannian metric $<,>$ on $M$, that we lift to $\Cov{M}$, such that
the pair $(f,<,>)$ is Morse Smale on $\Cov{M}$. In this situation, the
unstable manifold of a critical point $b$ of index $1$ of $f$ is a path
$\gamma_{b}$ going to $-\infty$ on both ends in $\Cov{M}$. We call the
restriction of such a path above a level $h$ a Morse-Novikov step
relative to $h$, and define the space of Morse-Novikov $\Loops(f)$ in the
same way as before. In this simplified situation, this sums up to
letting~: $\Loops(f) = \varprojlim_{h}\zipped{\Loops(f)}{h}$ where
$$
\zipped{\Loops(f)}{h}=<b\in\Crit_{1}(f)\ | \ b=1 \text{ if }f(b)\leq h >
$$
is the group freely generated by the index $1$ critical points $b$ of $f$
where  $f(b)\geq h$.

Finally, recall from \cite{NovikovPi1} that the natural evaluation map to the
Novikov fundamental group
$$
\Loops(f) \xrightarrow{\eval}\pi_{1}(\Cov{M},[\alpha])
$$
is surjective.

\bigskip

The object of this section is to prove the following proposition~:
\begin{proposition}\label{prop:psi}
  The above evaluation map factors through $\Loops(\phi)$, i.e. there is
  a group morphism $\psi$ making the following diagram commutative~:
  $$
  \xymatrix{
    \Loops(f) \ar[d]_-{\psi}\ar@{>>}[r]^-{\eval} & \pi_{1}(\Cov{M},\Xi)\\
    \Loops(\phi)\ar[ur]_-{\eval} 
  }
  $$
  In particular, the evaluation $\Loops(\phi)\to\pi_{1}(\Cov{M},\Xi)$ is
  surjective.
\end{proposition}

To associate a Floer-Novikov loop to a Morse-Novikov loop, we will make
use of hybrid moduli spaces, that are built out of the space $\M(\emptyset,
\emptyset)$ of solutions, in the trivial homotopy class, of the Floer
equation in which the Hamiltonian term is truncated at both ends. More
precisely, this equation is the following~:
\begin{equation}
  \label{eq:saucisse_equation}
\frac{\partial u}{\partial s}+J_{\chi(s),t}(u)
\Big(\frac{\partial u}{\partial t}-\chi_{R}(s)X_{t}(u)\Big) = 0.
\end{equation}
where
$$
\chi_{R}(s) = \chi(s-R)\chi(-s-R)
$$
is a smooth function such that $\chi_{R}(s)=1$ for $|s|\leq R-1$ and
$\chi_{R}(s)=0$ for $|s|\geq R$. Here $R$ is a non negative number that is part
of the unknown.

A solution $(u,R)$ of this equation with finite energy has limits at both
ends ends, which are just points in $\Cov{M}$. In particular, it induces
a map from $\S^{2}$ to $\Cov{M}$. We denote by $\M(\emptyset,\emptyset)$
the space of couples $(u,R)$ satisfying \eqref{eq:EnergySaucisse} such
that
\begin{enumerate}
\item
  $u$ has finite energy,
\item
  as a map from $\S^{2}$ to $\Cov{M}$, $u$ is in the trivial homotopy
  class.
\end{enumerate}

For a generic choice of $J$ this is a smooth manifold with boundary, of
dimension $n+1$. The boundary is given by the condition $R=0$, and
consists in the constant maps $u:\R\times\S^{1}\to \Cov{M}$.

The energy sub-levels are compact up to breaks, and we still denote
the space obtained by compactifying all the energy sub-levels  by
$\M(\emptyset,\emptyset)$.

Given an index $1$ critical point $b$ of $f$, the hybrid spaces we are
interested in are the following~:
$$
\M(b,\emptyset)=\{(u,R)\in\M(\emptyset,\emptyset),u(-\infty)\in W^{u}(b)\}
$$
and
$$
\M(b,\{f\geq h\})=\{(u,R)\in
\M(b,\emptyset), f(u(+\infty))\geq h\}.
$$

To control the compactness of such spaces with respect to depth, we again
need an energy-depth estimate for them, which is the object of the next
section.

\subsection{Energy/depth  estimate on $\M(\emptyset,\emptyset)$}

The same computations as for the augmentation curves shows that for
$(u,R)\in\M(\emptyset, \emptyset)$, we have
\begin{equation}
  \label{eq:EnergySaucisse}
E(u)=\iint_{A_{-}}H_{t}(u)|\chi_{R}'(s)|dsdt
-\iint_{A_{+}}H_{t}(u)|\chi_{R}'(s)|dsdt
\end{equation}
where $A_{\pm} = [\pm (R-1),\pm R]\times\S^{1}$.

We now want to turn this average estimate into a pointwise estimate at
both ends. Recall that $\lambda$ was defined by the relation
$[\alpha]=\lambda[\alpha_{0}]$ in \eqref{eq:lambda}.

\begin{lemma}\label{lem:EnergyDDepthEstimate}
  There is a positive constant $A'$ depending only on $M$, $\omega$, $J$
  and $\Xi_{0}$ (but not on $\lambda$), and a constant $B'$ that may
  depend also on $\lambda$ and $\phi$, such that for all maps
  $u\in\M(\emptyset,\emptyset)$~:
  $$
  f(u(-\infty))-f(u(+\infty))
  \geq
  E(u) (\frac{1}{2}-\lambda A') - B'
  $$
\end{lemma}
The proof proceeds along the same lines as the proof of lemma
\ref{lem:EnergyDepthEstimate} and is left to the reader.

In particular, this lemma implies that the relative height
$(f(u(-\infty))-f(u(+\infty)))$ is proper on $\M(\emptyset,\emptyset)$.

\subsection{Exploring boundary components of $\M(b,\{f\geq h\})$}

Fix some $b\in\Crit_{1}(f)$.

To each level $h\in\R$, is associated the moduli space
$$
\M(b,\{f\geq h\})=
\{(u,R)\in\M(b,\emptyset),
f(u(+\infty))\geq h\}.
$$

For a generic choice of $h$, it is a smooth $2$ dimensional manifold with
corners, whose boundary is given by the conditions
\begin{itemize}
\item 
  $f(u(+\infty))=h$
\item
  or $R=0$ (which correspond to the case when $u$ is a
  constant map),
\item
  or the configuration is broken at an intermediate orbit (recall there
  is no index 0 Morse critical point).
\end{itemize}


Exploring boundary components of $2$ dimensional moduli spaces by means
of ``crocodile walks'', as explained in \cite{FloerPi1} adapts
straightforwardly to the current situation. Since the involved
degenerations are slightly different, we still recall it briefly below, and refer
to \cite{FloerPi1} for a more detailed discussion.

We first need a description of the boundary. It can be described as
\begin{equation}
  \partial\M(b,\{f\geq h\})=B_{1}\cup B_{2} \cup B_{3}\cup B_{4}
\end{equation}
with
\begin{enumerate}
\item 
  $B_{1}=  \M(b,\{f=h\})$
\item
  $B_{2}=\big(W^{u}(b)\cap\{f\geq h\}\big)$
\item
  $B_{3}=\bigcup_{|y|=1}\M(b,y)\times\M(y,\{f\geq h\})$.
\item
  $B_{4}=\bigcup_{|x|=0}\M(b,x)\times\M(x,\{f\geq h\})$.
\end{enumerate}
Here, the set $B_{2}$ corresponds the condition $R=0$ and consists in the
arc of the unstable manifold of $b$ that lies above $h$ (in which each
point  $p$ is seen as the piece of Morse flow line from $b$ to $p$,
followed by the constant map $\R\times\S^{1}\to\{p\}\subset\Cov{M}$). The spaces
$B_{3}$ and $B_{4}$ correspond to configurations that are broken at an
intermediate orbit $z$, and cover all the possible indices for $z$, since
$\M(b,z) \neq \emptyset$ requires $|z|\leq 1$, and
$\M(z,\emptyset)\neq\emptyset$ requires $|z|\geq 0$.

The configurations in this boundary all undergo a degeneracy~: $\{f=h\}$
for $B_{1}$, $R=0$ for $B_{2}$, a Floer break for $B_{3}$ and $B_{4}$. We
will say that configurations in $B_{1}$ and $B_{3}$ undergo a lower
degeneracy, and those in $B_{2}$ and $B4$ an upper degeneracy.

Contained in this boundary are the ``corners''
$$
C=\partial B_{1}\cup\partial B_{2}\cup\partial B_{3}\cup\partial B_{4}
=\bigcup_{i\neq j}B_{i}\cap B_{j}.
$$
More explicitly, we let
\begin{equation}
  \label{eq:corners}
  C = C_{1}\cup C_{2} \cup C_{3}
\end{equation}
with
\begin{enumerate}
\item 
  $C_{1} = B_{1}\cap B_{2} = W^{u}(b)\cap\{f=h\} = \{p_{-},p_{+}\}$,
\item
  $C_{2} = B_{1}\cap B_{3} = \bigcup_{|y|=1}\M(b,y)\times\M(y,\{f=h\})$,
\item
  $C_{3} = B_{3}\cap B_{4} = \bigcup_{\substack{|y|=1\\|x|=0}}
  \M(b,y)\times\M(y,x)\times\M(x,\emptyset)$,
\end{enumerate}
all the other intersections being empty. Here, $p_{\pm}$ are the two
intersection points of $W^{u}(b)$ with the level $\{f=h\}$.

Observe that the configurations in $C$ are exactly those undergoing 2
degeneracy, which are always a lower one and an upper one.

The last required ingredient are ``gluing'' maps on the boundary of
 $1$-dimensional moduli spaces.

\begin{proposition}
  For a generic choice of $h$, there are maps
  \begin{align}
    \label{eq:gluing_byx}
    \M(b,y)\times\M(y,x)\times[0,\epsilon)&\to \M(b,x) \\
    \label{eq:gluing_yxh}
    \M(b,y)\times\M(y,\{f=h\})\times[0,\epsilon)&\to \M(b,\{f=h\})\\
    \label{eq:gluing_bRhh}
    \{p_{-},p_{+}\}\times[0,\epsilon)&\to \M(b,\{f= h\})\\
    \label{eq:gluing_yh}
    \M(y,\{f=h\})\times[0,\epsilon)&\to \M(y,\{f\geq h\})\\  
    \label{eq:gluing_bRRh}
    \{p_{-},p_{+}\}\times[0,\epsilon)&\to \M(y,\{f\geq h\})
  \end{align} 
  which are local homoeomoprhisms near the boundary points. These maps
  will be called ``gluing'' maps (although the three last ones do not
  glue two broken pieces together).
\end{proposition}

\begin{proof}
  The maps in \eqref{eq:gluing_byx} and \eqref{eq:gluing_yxh} are cutout
  from the usual Floer gluing maps by the relevant incidence conditions~:
  for a generic choice of data, they inherit all their properties from
  the original ones.

  The map in \eqref{eq:gluing_bRhh} resolves the $R=0$ condition keeping
  the $\{f=h\}$ condition~: the existence of such a map is obtained form
  the fact that the constants are regular values of the Floer equation
  \eqref{eq:saucisse_equation} (cf \cite{FloerPi1} for instance) and from
  the genericity assumption on $h$.    

  The map in \eqref{eq:gluing_yh} resolves the $f=h$ condition in
  $\M(y,\{f\geq h\})$~: it is again derived from the transversality
  assumption of the evaluation map and the level $h$.

  Finally, \eqref{eq:gluing_bRRh} resolves the $\{f=h\}$ along the
  unstable manifold of $b$, and is derived from the asumption that $h$ is
  a regular level for $f$.
\end{proof}

The gluing maps induce two involutions maps $C\xrightarrow{\glueup}C$ and
$C\xrightarrow{\gluelo}C$, defined by resolving the upper or lower
degeneracy and keeping the other~: in each case, the corner configuration
is seen as one end of a space $B_{i}$ ($i=1,3$ for $\gluelo$ and $i=2,4$
for $\glueup$), and the map assigns the other end. 
 
For finiteness reasons, alternating composition of $\gluelo$ and $\glueup$
then has to loop, and defines a sequence of components
$(\sigma_{1},\dots,\sigma_{k})$ of the spaces $B_{1},...,B_{4}$ with
alternating parity.

Observe now that an odd term $\sigma_{2i+1}$ in this sequence is either
\begin{enumerate}
\item
  $W^{u}(b)\cap\{f\geq h\}$, which can appear at most once,
\item
  or a path $(\beta_{i},\alpha_{i,t})$ where $\beta_{i}\in\M(b,y_{i})$ is
  fixed and $\alpha_{i,t}$ describse a component of $\M(y_{i},\{f\geq h\})$.
\end{enumerate}

In particular, but for the special step associated to $W^{u}(b)$, the
$\alpha_{i,t}$ form a sequence of consecutive Floer steps, and define an
element $\gamma\in\zipped{\Loops(\phi)}{h}$.

Notice moreover that each element $u\in\M(b,\{f\geq h\})$ comes with a
preferred path $\gamma_{u}$ joining $b$ to $\eval(u)$~: away from the
boundary, it is defined as the concatenation of
\begin{itemize}
\item 
  the piece of Morse flow line from $b$ to $u(-\infty)$, parametrized by the
  value of $f$,  
\item
  and the restriction of $u$ to the real line
  $\R\times\{0\}\subset\R\times\S^{1}$, parametrized by the energy of $u$.
\end{itemize}
Using Moore paths, one easily checks that this definition extends
continuously to the boundary.

In particular, when $u$ describes all the components of the sequence
$(\sigma_{1}, \dots , \sigma_{k})$ one after the other,
\begin{itemize}
\item 
 the points $\eval(u)$ describe a continuous loop in
 $\Cov{M}/\sublevel{\Cov{M}}{h}$,  which is the concatenation of
 evaluation of the Floer loop $\gamma$ and the arc defined by $W^{u}(b)$,
\item
 the paths $\gamma_{u}$ describe a continuous $S^{1}$ family of paths
 that all start at $b$~: they fill a disc, whose boundary is the above
 loop.
\end{itemize}
In particular, this proves that the Morse step associated to $b$ and the
Floer loop $\gamma$ are homotopic.

\subsection{Proof of proposition \ref{prop:psi} and theorem \ref{thm:EvalOntoPi1}}

Applying the above construction to the boundary component of
$\M(b,\{f\geq h\})$ that contains the component associated to
$\zipped{W^{u}(b)}{h}$, we obtain a Floer loop $\psi_{h}(b)$, whose
evaluation in $\Cov{M}/\sublevel{\Cov{M}}{h}$ is homotopic to $W^{u}(b)$. 

Repeating this for each $b$, we get a morphism
$$
\zipped{\Loops(f)}{h}\xxto{\psi_{h}}\zipped\Loops{\phi}{h}
$$
through which the
evaluation to $\pi_{1}(\Cov{M}/\sublevel{\Cov{M}}{h})$ factors.

These maps are compatible with the restrictions $\zip{h'}{h}$, and
passing to the limit, we get a morphism $\psi$ making the following
diagram
$$
\xymatrix{
    \Loops(f) \ar[d]_-{\psi}\ar@{>>}[r]^-{\eval} & \pi_{1}(\Cov{M},\Xi)\\
    \Loops(\phi)\ar[ur]_-{\eval}
}
$$
commutative. This proves the proposition \ref{prop:psi}, and hence
theorem \ref{thm:EvalOntoPi1}.


\begin{thebibliography}{00}


\bibitem{AL1994}
  Holomorphic curves in symplectic geometry.
  Edited by Michèle Audin and Jacques Lafontaine.
  Progress in Mathematics, 117. Birkhäuser Verlag, Basel, 1994. xii+328 pp.
  ISBN: 3-7643-2997-1
  
\bibitem{FloerPi1}%
  {\sc J.-F.  Barraud},  A Floer
  fundamental group, Annales Scientifiques de l'École Normale Supérieure,
  Elsevier Masson, 2018, 51 (3), pp.773-809.

\bibitem{NovikovPi1}%
  {\sc J.-F.  Barraud, A. Gadbled, R. Golovko, H.V. Le}, A Novikov fundamental group
  International Mathematics Research Notices, Oxford University Press (OUP), 2019

\bibitem{Mihai2009}%
  {\sc M. Damian}, Constraints on exact Lagrangians
  in cotangent bundles of manifolds fibered over the circle. Comment.
  Math. Helv. 84(4), 705–746 (2009).





\bibitem{Floer1}
{\sc A. Floer}, {Cuplength estimates on Lagrangian intersections}, 
Comm. Pure Appl. Math. 42 (1989), no. 4, 335--356.

\bibitem{Floer2}
  {\sc A. Floer}, { Morse theory for Lagrangian intersections},
  J. Differential Geom. 28 (1988), no. 3, 513--547.

\bibitem{Floer3}
  {\sc A. Floer},  {Witten's complex and infinite-dimensional Morse theory},
  J. Differential Geom. 30 (1989), no. 1, 207--221.

\bibitem{Floer4}
  {\sc A. Floer}, {Symplectic fixed points and holomorphic spheres},
  Commun. Math. Phys. 120 (1989), 575--611.

\bibitem{Agnes2009}%
  {\sc A. Gadbled}, Obstructions to the
  existence of monotone Lagrangian embeddings into cotangent bundles of
  manifolds fibered over the circle, Annales de l'Institut Fourier, 59
  (2009), no. 3, pp 1135-1175.

\bibitem{Gompf}
  {\sc R.  E.  Gompf}.
  A new construction of symplectic manifolds.
  Annals of Mathematics,142(3):527–595, November 1995.

  
\bibitem{Gromov1985}%
  {\sc M. Gromov}, Pseudo holomorphic curves in symplectic manifolds.
  Invent. Math. 82 (1985), no. 2, 307–347.


\bibitem{LeOno1995}%
  {\sc H. V. L\^e and K. Ono}, Symplectic fixed
  points, the Calabi invariant and Novikov homology, Topology 34 (1995),
  155-176.

\bibitem{PSS}%
  {\sc S. Piunikhin, D. Salamon and M. Schwarz}, Symplectic Floer-Donaldson
  theory and quantum cohomology, Contact and symplectic geometry (Cambridge,
  1994) (Cambridge University Press, Cambridge, 1996) 171–200.

  





\bibitem{Sikorav1986}%
  {\sc J.-C. Sikorav}, Un problème de
  disjonction par isotopie symplectique dans un fibré cotangent, Annales
  scientifiques de l'École Normale Supérieure, Série 4 : Tome 19 (1986)
  no. 4 , 543-552.

\bibitem{Sikorav1987}%
  {\sc J.-C. Sikorav}, Points fixes de
  diff\'eomorphismes symplectiques, intersections de sous-vari\'et\'es
  lagrangiennes, et singularit\'es de un-formes ferm\'ees, Th\'ese de
  Doctorat d’Etat Es Sciences Math\'ematiques, Universit\'e Paris-Sud,
  Centre d’Orsay, 1987.



\end{thebibliography}
\end{document}